\documentclass{amsart}
\RequirePackage{amsmath,amssymb,array,amsthm,stmaryrd}
  \numberwithin{equation}{section}

  \RequirePackage{cite}
  \RequirePackage{hyperref}
  \RequirePackage{doi}

\RequirePackage{microtype}

  \DeclareMathOperator{\Hom}{Hom}
  
  \DeclareMathOperator{\maps}{maps}

  \DeclareMathOperator*{\holim}{holim}
  
  \DeclareMathOperator*{\hofib}{hofib}

  \newtheorem{theorem}{Theorem}[section]
  \newtheorem*{theorem*}{Theorem}
  \newtheorem{proposition}[theorem]{Proposition}
  \newtheorem{lemma}[theorem]{Lemma}
  \newtheorem{corollary}[theorem]{Corollary}
  \newtheorem{conjecture}[theorem]{Conjecture}

  \theoremstyle{definition}
  \newtheorem{remark}[theorem]{Remark}
  
  \newtheorem{example}[theorem]{Example}
  \newtheorem{definition}[theorem]{Definition}

    \theoremstyle{definition}
    
    \newtheorem*{q*}{Q}

\RequirePackage{multicol}
\RequirePackage{longtable}
\RequirePackage{enumitem}
\RequirePackage{xcolor}
\RequirePackage{hyperref}
  \definecolor{link}{rgb}{0,0,0.5}
  \hypersetup{
     unicode,
     colorlinks=true,
     citecolor=link,
     filecolor=link,
     linkcolor=link,
     urlcolor=link
  }

\usepackage[all,cmtip]{xy}

\newcommand{\T}{\mathcal{T}}

\newcommand{\A}{\mathcal{A}}

\newcommand{\R}{\mathbb{R}}
\renewcommand{\C}{\mathbb{C}}

\DeclareMathOperator{\embtr}{Emb_{TR}}
\DeclareMathOperator{\embl}{Emb_{Lag}}

\DeclareMathOperator{\gr}{{Gr}}
\DeclareMathOperator{\GL}{{GL}}
\DeclareMathOperator{\imm}{Imm}
\DeclareMathOperator{\immf}{Imm^f}
\DeclareMathOperator{\immfa}{Imm^f_{\A}}
\DeclareMathOperator{\emb}{Emb}

\DeclareMathOperator{\embfa}{Emb^f_{\A}}
\DeclareMathOperator{\bs}{bs}
\DeclareMathOperator{\man}{\mathsf{Man_m}}
\renewcommand{\top}{\mathsf{Top}}
\DeclareMathOperator{\disc}{\mathsf{Disc}_\infty}
\DeclareMathOperator{\psh}{\mathsf{PSh}}

\begin{document}

\title{An application of the $h$-principle to manifold calculus}
\author{Apurva Nakade}
\address{Department of Mathematics, University of Western Ontario, London, ON, Canada.}
\email{anakade@uwo.ca}
\urladdr{http://apurvanakade.github.io/}

\begin{abstract}
  	Manifold calculus is a form of functor calculus that analyzes contravariant functors from some categories of manifolds to topological spaces by providing \emph{analytic approximations} to them. In this paper, using the technique of the $h$-principle, we show that for a symplectic manifold $N$, the analytic approximation to the Lagrangian embeddings functor $\mathrm{Emb}_{\mathrm{Lag}}(-,N)$ is the totally real embeddings functor $\mathrm{Emb}_{\mathrm{TR}}(-,N)$.
    More generally, for subsets $\A$ of the $m$-plane Grassmannian bundle $\gr(m,TN)$ for which the $h$-principle holds for $\A$-directed embeddings, we prove the analyticity of the $\A$-directed embeddings functor $\emb_{\A}(-,N)$.
\end{abstract}

\maketitle

\section{Introduction}
A symplectic manifold is a smooth manifold, $N$, equipped with a closed, non-degenerate differential 2-form; see Section~\ref{sec:Lagrangian_embeddings} for details.
One very important invariant associated to $N$ is its Fukaya category whose objects are Lagrangian submanifolds of $N$.
The Fukaya category appears in the statement of the homological mirror symmetry conjecture of Kontsevich.
For this and other reasons, the study of the space of Lagrangian submanifolds of symplectic manifolds is of great interest in symplectic geometry.

One very important example of a symplectic manifold is the cotangent bundle of a smooth manifold.
Much of the study of the space of Lagrangian submanifolds of the cotangent bundle has centered on Arnold's Nearby Lagrangian Conjecture.
Recent attempts to prove this conjecture combine techniques from the classical theory of J-holomorphic curves with modern homotopy theoretic techniques~\cite{Kragh_Parametrized_Ring_Spectra_}.
A simplified version of this conjecture can be stated in the following way.
Let $L$ and $L'$ be closed, simply-connected, smooth manifolds of the same dimension.

\begin{conjecture}[Nearby Lagrangian Conjecture]
  The space of unparametrized Lagrangian embeddings of $L$ into $T^*L'$ is non-empty if and only if $L$ is diffeomorphic to $L'$, in which case it is contractible.
\end{conjecture}
It is known that if $L$ is a Lagrangian submanifold of $T^*L'$ then the natural projection map is a simple homotopy equivalence between $L$ and $L'$~\cite{Abouzaid_Nearby_Lagrangians_with_vanishing_Maslov_}, \cite{Kragh_Parametrized_Ring_Spectra_},
\cite{Abouzaid_Kragh_Simple_homotopy_equivalence_of_nearby_Lagrangians}.
The conjecture has been shown to be true for $L' = S^2$~\cite{Hind_Lagrangian_spheres_inS2xS2_},
\cite{Hind_Lagrangian_unknottedness_in_Stein_surfaces},
\cite{Hind_Pinsonnault_Wu_Symplectormophism_groups_of_non_compact_manifolds_}.
In this paper, we initiate a program to apply homotopy theoretic methods coming from manifold calculus to study embedding spaces coming from symplectic geometry.

Let $M$ and $N$ be smooth manifolds and let
$\emb(M,N)$ be the space of smooth embeddings of $M$ into $N$.
Manifold calculus, as first defined in \cite{Weiss_Embeddings_from_} and later reformulated in \cite{Pedro_Manifold_calculus_and_homotopy_sheaves}, studies contravariant homotopy functors on manifolds, such as $F(M) = \emb(M,N )$. To such a functor $F$ it associates a tower of fibrations
\begin{align*}
   \cdots \longrightarrow \T_k F(M) \longrightarrow \T_{k-1} F(M) \longrightarrow \cdots \longrightarrow \T_{1} F(M)
\end{align*}
where the approximation $\T_k F(M)$ is the homotopy limit of $F$ evaluated on the category of open subsets of $M$ diffeomorphic to $k$ or fewer copies of $\R^m$, where $m$ is the dimension of $M$.
This tower can be thought of as a ``Taylor tower'' of $F$.
Denote by $\T_\infty F(M)$ the homotopy inverse limit of $\T_k F(M)$.

There also exist compatible maps $F(M) \rightarrow \T_k F(M)$ which induces a map $F(M) \rightarrow \T_\infty F(M)$. When this map is a homotopy equivalence, one says that the Taylor tower converges to $F$ and the functor $F$ is \emph{analytic}.
A deep and important convergence result for the functor $F(M) = \emb(M,N)$ states the following.
\begin{theorem}
  [Goodwillie--Weiss \cite{Goodwillie_Weiss_Embeddings_from_}, Goodwillie--Klein \cite{Goodwillie_Klein_Multiple_disjunction_}]
  \label{thm:GWK-Introduction}
  The embeddings functor $\emb(M,N)$ is analytic if $\dim(N) - \dim(M) \ge 3$.
\end{theorem}

In Section~\ref{sec:Lagrangian_embeddings}, we prove that manifold calculus sees the ``flexible side'' of symplectic geometry that can be studied using the technique of the $h$-principle.
Let $(N,\omega)$ be a symplectic manifold and let $J$ be any almost complex structure on it which is compatible with $\omega$.
Let  $\embl(M,N)$ and $\embtr(M,N)$ be the spaces of Lagrangian and totally real embeddings, respectively, of $M$ into $N$ (see Section~\ref{sec:Lagrangian_embeddings} for definitions).
\begin{theorem}
  When $\dim(N) - \dim(M) \ge 3$, the analytic approximation of the Lagrangian embeddings functor $\embl(M,N)$ is weakly homotopy equivalent to the totally real embeddings functor $\embtr(M,N)$.
\end{theorem}

In general, the space of Lagrangian embeddings is not the same as the space of totally real embeddings.
As such, this result provides an example of a non-analytic functor.
The author is unaware of any other examples of non-analytic functors in the present literature on manifold calculus.

We prove this theorem using the technique of the $h$-principle for directed embeddings.
Let $\gr(m,TN)$ be the $m$-plane Grassmannian bundle over $N$ and let $\A$ be a subset of $\gr(m,TN)$.
An embedding $e:M \hookrightarrow N$ is called $\A$-directed if the image of the naturally induced map $\gr(m,e):M \rightarrow \gr(m,TN)$ lies in $\A$.
We say that the $h$-principle holds for $\A$-directed embeddings if an arbitrary embedding $e$ can be perturbed to get an $\A$-directed embedding, provided we are also given a tangential homotopy connecting $De$ to $\A$, where $De$ is the derivative of $e$; see Section~\ref{sec:directed_immersions_and_embeddings}.
In Section~\ref{sec:Main_theorems} we prove the following result.

\begin{theorem}
  \label{thm:mainThm-Introduction}
  If $\dim(N) - \dim(M) \ge 3$ and the $h$-principle holds for $\A$-directed embeddings, then the space $\emb_\A(M,N)$ is weakly homotopy equivalent to its analytic approximation $\T_\infty\emb_\A(M,N)$.
\end{theorem}

Theorem~\ref{thm:mainThm-Introduction} recovers the Theorem~\ref{thm:GWK-Introduction} of Goodwillie--Weiss, Goodwillie--Klein for  $\A = \gr(m,TN)$. So it can be seen as a generalization of their result.

In the case when $M = S^1$, the invariants coming from manifold calculus for $\emb(S^1,\R^n)$ have been shown to be related to finite type knot invariants~\cite{Sinha_Embedding_calculus_knot_invariants_are_of_finite_type}.
In higher dimensions, manifold calculus has been used to produce loop space structures on embedding spaces of discs~\cite{Boavida_Weiss_Spaces_of_smooth_embeddings_}, and for proving finiteness results about homotopy groups of automorphisms spaces of discs~\cite{Kupers_Some_finiteness_results_for_groups_of_automorphisms_of_manifolds}.
We hope to generalize these results to embedding spaces that arise from the $h$-principle.

\subsection{Notation and Conventions.} Throughout the paper we will work in the category of smooth manifolds without boundary.
The mapping spaces are endowed with the weak $ C^\infty$ topology.
$M$ and $N$ will denote smooth manifolds of dimensions $m$ and $n$ respectively.
Throughout the paper we will assume that $m < n$.
We will use the terms space and topological space interchangeably.
All categories are enriched over spaces and likewise for all the categorical constructions.

\subsection*{Acknowledgements.} The author would like to thank Nitu Kitchloo and Dan Christensen for several helpful discussions about this project, and the organizers and speakers of the Alpine Algebraic and Applied Topology Conference, where the author first learned about manifold calculus. The author would like to thank the anonymous referees for their valuable comments which helped to improve the manuscript.
This is a pre-print of an article published in \emph{Journal of Homotopy and Related Structures}. The final authenticated version is available at \cite{Nakade_An_application_of_the_h-principle_to_manifold_calculus}.

\section{Manifold Calculus}
\label{sec:Background_Manifold_Calculus}

We start by recalling basic definitions of manifold calculus from \cite{Pedro_Manifold_calculus_and_homotopy_sheaves}.
Denote by $\top$ the category of compactly generated Hausdorff spaces enriched over itself.
The category $\man$ is the $\top$-enriched category defined as
\begin{align*}
	Ob(\man)    & := \{ m \mbox{ dimensional smooth manifolds without boundary} \} \\
	{\man}(U,V) & :=  \emb(U,V)
\end{align*}
where $ \emb(U,V)$ is the space of embeddings $ U \hookrightarrow V$ topologized under the weak $ C^\infty$ topology \cite{book:Hirsch_Differential_Topology}.
\begin{definition}
  Define $\disc$ to be the full subcategory  of $\man$ consisting of manifolds diffeomorphic to a disjoint union of finitely many $\R^m$.
\end{definition}

The presheaf category $ \psh(\man)$ consisting of functors $\man^{\mathrm{op}} \rightarrow \top$ has a natural \emph{projective model structure} \cite{book:Hirschhorn_Model_Categories} induced by the model structure on $\top$.
The fibrations are object-wise fibrations, the weak equivalences are object-wise weak equivalences, and
the cofibrations are maps which satisfy the right lifting property with respect to trivial fibrations.

\begin{definition}
	\label{def:analytic_approximation}
	For a functor $ F$ in $\psh(\man)$, the \emph{analytic approximation} to $F$ is the functor $\T_\infty F$ in $\psh(\man)$ defined as the right derived Kan extension of $ F|_{\disc}$ along the inclusion $ \disc \hookrightarrow \man$.
	\begin{align*}
		{\xymatrix{
		\disc \ar^{F|_{\disc}}[rr] \ar[d] &  & \top \\
		\man \ar@{-->}_{\T_\infty F}[urr]
		}}
	\end{align*}
	More explicitly, for a manifold $ M$ in $\man$,
	\begin{align*}
		 & \T_\infty F(M)
		:= \Hom _{\psh(\disc)}(Q\emb(-,M),F)
	\end{align*}
	where $Q\emb(-,M)$ is the cofibrant replacement of $ \emb(-,M)$ in $\psh(\man)$.
\end{definition}

\begin{definition}
	We say that a functor $ F$ in $\psh(\man)$ is \textit{analytic} if the natural map
	\begin{align*}
		\xymatrix{
		F \ar[r]^-{\simeq} & \T_\infty F
		}
	\end{align*}is a weak homotopy equivalence.
\end{definition}

\begin{example}
	\label{example:Analytic_functors}
	The following examples of analytic functors will be of use to us in the later sections.
	\begin{enumerate}
		\item By the formal properties of Kan extensions, it follows that
		      \begin{align*}
			      \T _ \infty \T_\infty F \simeq \T_\infty F
		      \end{align*}
		      for any functor $ F$ in $\psh(\man)$. Hence, an analytic approximation $ \T_\infty F$ is itself always analytic.

		\item
		      For a topological space $X$, the functor $ \maps(-,X) $ of all \emph{continuous} maps into $X$ is analytic, \cite[Example 2.4]{Weiss_Embeddings_from_}, \cite[Chapter 10.2]{book:Munson_Cubical_Homotopy_Theory}.

		\item
		      Let $ \imm(M,N)$ denote the space of immersions of $ M$ into $ N$. For $ n > m$, the functor $\imm(-,N)$ in $\psh(\man)$ is analytic, \cite[Example 2.3]{Weiss_Embeddings_from_}.
	\end{enumerate}
\end{example}

Manifold calculus was introduced to study embedding spaces of manifolds. One of the deepest theorems in manifold calculus states the following.
\begin{theorem}[Goodwillie--Weiss \cite{Goodwillie_Weiss_Embeddings_from_}, Goodwillie--Klein \cite{Goodwillie_Klein_Multiple_disjunction_}]
	\label{theorem:Goodwillie_Weiss_theorem}
	If $ n - m \ge 2$, then the functor $ \emb(-,N)$ in $\psh(\man)$ is analytic.
\end{theorem}

\section{Directed Immersions and Embeddings}
\label{sec:directed_immersions_and_embeddings}
In this section, we recall the notions of directed immersions and embeddings and the corresponding $h$-principles as defined in \cite{book:Gromov_Partial_Differential_Relations},
\cite{book:Eliashberg_h_principle},
\cite{book:Spring_Convex_Integration_Theory}.

The homotopy principle ($h$-principle) is a very general method for reducing problems from geometry to homotopy theory.
The $h$-principle first appeared in
the Whitney--Graustein Theorem~\cite{Whitney_On_regular_closed_curves_in_the_plane},
the Smale--Hirsch Immersion Theorem~\cite{Smale_Classification_of_Immersions_of_Spheres_},
and the Nash--Kuipers Embedding Theorem~\cite{Nash_C1_isometric_imbeddings},
\cite{Nash_The_imbedding_problem_for_Riemannian_manifolds}, \cite{Kuipers_On_C1_isometric_imbeddings_I}.
These theorems provide homotopy theoretic descriptions of the spaces of immersions and embeddings of manifolds in Euclidean spaces.
Their proof techniques were vastly generalized by Gromov~\cite{book:Gromov_Partial_Differential_Relations}
to a general framework for finding solutions to underdetermined differential relations.

The following is a very simplified idea of the general $h$-principle; for details see~\cite{book:Eliashberg_h_principle}. Suppose we want to find a smooth map $e:M \rightarrow N$ satisfying a first order differential relation
\begin{align}
  \label{eq:PDR}
  P(e, De) = 0,
\end{align}
where $De$ is the differential of $e$.
We consider instead the \emph{decoupled} differential relation
\begin{align}
  \label{eq:PDR1}
  P(e, F) &= 0,
\end{align}
where $F$ is any bundle map from $TM$ to $TN$.
The solutions of equation~\eqref{eq:PDR} are called \emph{genuine} solutions and solutions of equation~\eqref{eq:PDR1} are called \emph{formal} solutions.
A formal solution $(e,F)$ is a genuine solution if $F = De$.
It is usually possible to describe the space of formal solutions as the space of sections of certain naturally occuring bundles constructed from $TM$ and $TN$ (see below).
We say that the $h$-principle holds for the partial differential relation $P$ if the space of formal solutions is weakly homotopy equivalent to the space of genuine solutions.
The weak homotopy equivalence is proven by showing that any continuously varying family of formal solutions $(e_s,F_s)$ can be perturbed to a continuously varying family of {genuine} solutions $(e'_s, F'_s)$.
Thus, when the $h$-principle holds, the problem of finding the homotopy type of the space of genuine solutions is reduced to the simpler problem of finding the homotopy type of the space of formal solutions.

For a real vector space $V$, let $\gr(m,V)$ be the $m$-plane Grassmannian, i.e., the space of $m$-dimensional subspaces of $V$.
Let $\bs:\gr(m,TN) \rightarrow N$ be the \emph{$ m$-plane Grassmannian bundle} over $N$ whose fiber over a point $q$ in $N$ is $\gr(m,T_qN)$, as described by the following pullback diagram.
\begin{align*}
	\xymatrix{
      \gr(m,T_q N) \ar[r] \ar[d] & \gr(m,TN) \ar^{\bs}[d] \\
        {\ast} \ar^-{q}[r] & N}
\end{align*}
An immersion $e: M \looparrowright N$ induces a map $\gr(m,e): M \rightarrow \gr(m,TN)$ sending a point $p$ in $M $ to $De(T_p M)$.\\

\textbf{For the rest of this section, fix a manifold $N$ and a subset ${\A}$ of $\gr(m,TN)$.}

\begin{definition}
	\label{def:A_directed_Embedding}
	We say that an immersion (or an embedding) $ e: M \rightarrow N$ is \emph{$\A$-directed} if the image of $\gr(m,e)$ lies in $\A$.
\end{definition}
Denote by $ \imm_{\A}(M,N)$ and $ \emb_{\A}(M,N)$ the space of {\it $\A$-directed immersions} and {\it $\A$-directed embeddings}, respectively.

\begin{definition}
  \label{def:formalImmersion}
	A {\it formal immersion} from $M$ to $N$ is a vector bundle monomorphism from $TM$ to  $TN$.
  A formal immersion $F:TM \rightarrow TN$ is \emph{$\A$-directed} if the image of $F$ lies in $\A$.
\end{definition}
Denote the space of formal immersions and the space of formal $\A$-directed immersions by $\immf(M,N)$ and $\immfa(M,N)$, respectively.
There are natural inclusion maps $\imm(M,N) \hookrightarrow \immf(M,N)$ and $\imm_{\A}(M,N) \hookrightarrow \imm_{\A}^f(M,N)$ sending an immersion $e:M \rightarrow N$ to its differential $De:TM \rightarrow TN$.

\begin{definition}
	We say that the \emph{$h$-principle holds for $\A$-directed immersions} if the inclusion
	$\imm_{\A}(M,N) \hookrightarrow \immfa(M,N)$
	is a weak homotopy equivalence for all manifolds $M$ in $\man$.
\end{definition}

\begin{definition}
  \label{def:directedEmbedding}
 A \emph{formal $\A$-directed embedding} is a pair $(e,\gamma)$ where $e$ is an embedding in $\emb(M,N)$ and $\gamma$ is a smooth map $\gamma : [0,1] \rightarrow \maps(M,\gr(m,TN))$ satisfying \begin{enumerate}
    \item $\bs \circ \gamma(t) = e$ for all $t \in [0,1]$,
    \item $\gamma(0) = \gr(m,e)$,
    \item image of $\gamma(1)$ lies in $\A$.
  \end{enumerate}
The path $\gamma$ is called a \emph{tangential homotopy} lying over $e$.
\end{definition}
Denote the space of formal $\A$-directed embeddings by $\emb_{\A}^f(M,N)$.
There is a natural inclusion $\emb_{\A}(M,N) \hookrightarrow \emb_{\A}^f(M,N)$ which sends an embedding $e: M \hookrightarrow N$ to the constant tangential homotopy at $\gr(m,e)$ lying over $e$.

\begin{definition}
	\label{def:h_principle_directed_Embedding}
	We say that the \emph{$h$-principle holds for $\A$-directed embeddings} if the inclusion
	$\emb_{\A}(M,N) \rightarrow \emb_{\A}^f(M,N)$
	is a weak homotopy equivalence for all manifolds $M$ in $\man$.
\end{definition}


\section{Main Theorems}
\label{sec:Main_theorems}
In this section, we connect the theories of manifold calculus and the $h$-principle for directed embeddings.
For this section, fix a manifold $N$ and let $ \A$ be a subset of $\gr(m,TN)$ for which the $h$-principle holds for $\A$-directed embeddings. Further assume that the projection map $\bs:\A \rightarrow N$ is a fibration.

\begin{lemma}
  \label{lem:fibration}
  The projection map onto the first coordinate from $\embfa(M,N)$ to  $\emb(M,N)$ is a fibration.
\end{lemma}
\begin{proof}
  Let $X \hookrightarrow Y$ be a cofibration.
  From Definition~\ref{def:directedEmbedding} of $\embfa(M,N)$ it follows that the lifting problem for the square
  \begin{align*}
    \xymatrix{
    X \ar[r] \ar[d] & \embfa(M,N) \ar[d]\\
    Y \ar[r] \ar@{-->}[ru]& \emb(M,N)
    }
  \end{align*}
  is equivalent to producing compatible lifts for the following two squares.
  \begin{align*}
    \xymatrix{
    (X \times \{1\}) \times M
      \ar[r] \ar[d]
      & \A \ar[d]\\
    (Y \times \{1\}) \times M
      \ar[r] \ar@{-->}[ru]
      & N
    \\
    (X \times [0,1] \cup Y \times \{0,1\}) \times M
      \ar[r] \ar[d]
      & \gr(m,TN) \ar[d]\\
    (Y \times [0,1]) \times M
      \ar[r] \ar@{-->}[ru]
      & N
    }
  \end{align*}
  There exist compatible solutions for these lifting problems as the inclusions
  \begin{align*}
    (X \times \{1\}) &\longrightarrow (Y \times \{1\}) \\
    (X \times [0,1] \cup Y \times \{0,1\}) \times M &\longrightarrow (Y \times [0,1]) \times M
  \end{align*}
  are cofibrations and the maps $\A \rightarrow N$ and $\gr(m,TN) \rightarrow N$ are fibrations.
\end{proof}

\begin{theorem}
	\label{theorem:Main_Pullback_Diagram}
	For every manifold $ M$ in $\man$, the following commuting square is a homotopy pullback square,
  \begin{align}
		\xymatrix{
		\emb_{\A}(M,N) \ar^-{\gr(m,-)}[rr] \ar[d] && \maps(M,\A) \ar[d]  \\
		\emb(M,N) \ar^-{\gr(m,-)}[rr]            && \maps(M,\gr(m,TN))
		}
	\end{align}
  where the vertical maps are inclusions.
\end{theorem}
\begin{proof}
  Fix an embedding $e: M \hookrightarrow N$.
  We will show that the homotopy fiber
  $\hofib_e(\emb_{\A}(M,N) \rightarrow \emb(M,N))$ is weakly homotopy equivalent to homotopy fiber $\hofib_{\gr(m,e)}(\maps(M,\A) \rightarrow \maps(M,\gr(m,TN)))$.

  As the $h$-principle holds for $\A$-directed embeddings, the inclusion of $\emb_{\A}(M,N)$ into $\embfa(M,N)$ is a weak homotopy equivalence.
  The projection map from $\emb_{\A}(M,N)$ to $\emb(M,N)$ factors through $\embfa(M,N)$ giving us a weak homotopy equivalence
  \begin{align*}
    \hofib_e(\emb_{\A}(M,N) \rightarrow \emb(M,N))
    \overset{\simeq}{\longrightarrow}
    \hofib_e(\embfa(M,N) \rightarrow \emb(M,N)).
  \end{align*}
  As shown in Lemma~\ref{lem:fibration}, the map $\embfa(M,N) \rightarrow \emb(M,N)$ is a fibration. Hence, the homotopy fiber $\hofib_e(\embfa(M,N) \rightarrow \emb(M,N))$ is simply the fiber over $e$, which is the space of tangential homotopies lying over $e$ ending in $\A$:
  \begin{align}
        \label{eq:homotopyFiberDescription}
        \begin{split}
          \{ \gamma: [0,1] \rightarrow \maps(M,\gr(m,TN)) \mid  &\bs \circ \gamma = e, \\
          &\gamma(0) = \gr(m,e), \\
          &\gamma(1) \in \A \}.
        \end{split}
  \end{align}

  The homotopy fiber of the map $\maps(M,\A) \rightarrow \maps(M,\gr(m,TN))$ over $\gr(m,e)$ is the total homotopy fiber \cite[Chapter 3.4]{book:Munson_Cubical_Homotopy_Theory} of the following commutative square over $e$.
  \begin{align*}
    \xymatrix{
    \maps(M,\A) \ar[r] \ar[d]_-{\bs_1}
    &
      \maps(M,\gr(m,TN)) \ar[d]^-{\bs_2} \\
      \maps(M,N) \ar@{=}[r]
    &
      \maps(M,N)
    }
  \end{align*}
  As $\A \rightarrow N$ is a fibration, so are the vertical maps.
  The total homotopy fiber is then the homotopy fiber of the map $\bs_1^{-1}(e) \rightarrow \bs_2^{-1}(e)$ which is precisely the space described in \eqref{eq:homotopyFiberDescription}.
\end{proof}

\begin{corollary}
	If further the $h$-principle holds for $\A$-directed immersions, then for every manifold $ M$ in $\man$, the following commuting square is a homotopy pullback square,
  \begin{align}
    \label{eq:embeddings_immersions_pullback}
		\xymatrix{
		\emb_{\A}(M,N) \ar[r] \ar[d] & \imm_{\A}(M,N) \ar[d]  \\
		\emb(M,N) \ar[r]            & \imm(M,N)
		}
	\end{align}
  where all the maps are inclusions.
\end{corollary}

\begin{proof}
  We can extend the above square as follows.
  \begin{align}
		\xymatrix{
		\emb_{\A}(M,N) \ar[r] \ar[d]
      & \imm_{\A}(M,N)  \ar[d] \ar[rr]^-{\gr(m,-)}
      && \maps(M,\A) \ar[d]\\
		\emb(M,N) \ar[r]
      & \imm(M,N) \ar[rr]^-{\gr(m,-)}
      && \maps(M,\gr(m,TN))
		}
	\end{align}
  In Theorem~\ref{theorem:Main_Pullback_Diagram} we have shown that the larger square is a homotopy pullback square.
  By 2-out-of-3 property of homotopy pullbacks, it suffices to show that that the right square a homotopy pullback square.
  This square factors through the space of formal immersions as follows.
  \begin{align}
    \label{eq:immersionsSquare}
    \xymatrix{
    \imm_{\A}(M,N) \ar[r] \ar[d]
      & \immfa(M,N)  \ar[d] \ar[r]
      & \maps(M,\A) \ar[d]\\
    \imm(M,N) \ar[r]
      & \mathrm{Imm_{\gr(m,TN)}^f}(M,N) \ar[r]
      & \maps(M,\gr(m,TN))
    }
  \end{align}

  The map $\imm_{\A}(M,N) \rightarrow \immfa(M,N)$ is a weak homotopy equivalence as the $h$-principle holds for $\A$-directed immersions.
  The space $\mathrm{Imm_{\gr(m,TN)}^f}(M,N)$ is simply the space of vector bundle monomorphisms $TM \rightarrow TN$. Because of our assumption that $m < n$, the left lower horizontal map is a weak homotopy equivalence by the Smale--Hisrch Theorem~\cite{Smale_Classification_of_Immersions_of_Spheres_}, \cite{book:Eliashberg_h_principle}.
  Hence, the left square in \eqref{eq:immersionsSquare} is a homotopy pullback square.

  It follows from Definition~\ref{def:formalImmersion} that the maps $\immfa(M,N) \rightarrow \maps(M,\A)$ and $\mathrm{Imm_{\gr(m,TN)}^f}(M,N) \rightarrow \maps(M,\gr(m,TN))$ are both principal $\GL_m(\R)$ bundles.
  Hence, the right square in \eqref{eq:immersionsSquare} is a homotopy pullback square.
\end{proof}

\begin{lemma}\label{theorem:holim_lemma}
	For a small $I$-shaped diagram of analytic functors $F: I \rightarrow \psh(\man)$,
	the homotopy limit $ \holim _{i \in I} F_i$ is also analytic.
\end{lemma}
\begin{proof}
	For a diagram of analytic functors $ F: I \rightarrow \psh(\man)$ we have,
	\begin{align}
		(\mathcal{T}_\infty \holim _ I F_ i)(M)
    \label{eq:eq1}
		 & = \Hom _{\psh(\disc)}(Q\emb(-,M),\holim _ I F_ i)      \\
		\label{eq:eq2}
		 & \simeq \holim _ I \Hom _{\psh(\disc)}(Q\emb(-,M),F_ i) \\
     \label{eq:eq3}
		 & = \holim _ I (\mathcal{T}_\infty F_i)(M)         \\
		\label{eq:eq4}
		 & \simeq \holim _ I F_i(M)
	\end{align}
	where the equalities in \eqref{eq:eq1}  and \eqref{eq:eq3}  are by the definition of $\mathcal{T}_\infty$, the homotopy equivalence in \eqref{eq:eq2} follows from the universal property of enriched $ \holim$ and the homotopy equivalence in \eqref{eq:eq4} follows from the analyticity of $ F_i$.
\end{proof}

\begin{theorem}
	\label{theorem:Main_theorem}
	Let $ n - m \ge 3$. Let $\A$ be a subset of $\gr(m,TN)$ for which the $h$-principle holds for $\A$-directed embeddings. Further suppose that the projection map $\bs: \A \rightarrow N$ is a fibration.
	\begin{enumerate}
		\item The functor $ \emb_{\A}(-,N)$ in $\psh(\man)$ is analytic, i.e., the natural map
		      \begin{align*}
			      \xymatrix{
			      \emb_{\A}(M,N) \ar[r]^-{\simeq} & \mathcal{T}_\infty \emb _{\A} (M,N)
			      }
		      \end{align*}
		      is a weak homotopy equivalence for all manifolds $M$ in $\man$.
		\item Let $\A'$ be a subset of $\A$ such that the projection map $\bs:\A' \rightarrow N$ is a fibration and the inclusion $ \A' \hookrightarrow \A$ is a weak homotopy equivalence.
    Then there exists a weak homotopy equivalence
		      \begin{align*}
			      \T_\infty \emb_{\A'}(M,N) & \simeq \emb_{\A}(M,N)
		      \end{align*}
		      for all manifolds $M$ in $\man$.
	\end{enumerate}
\end{theorem}
\begin{proof}
	For $n - m \ge 3$, as mentioned in Example \ref{example:Analytic_functors} and Theorem~\ref{theorem:Goodwillie_Weiss_theorem} the three functors \begin{align*}
		\maps(-,\A)\mbox{, }  \maps(-,\gr(m,TN)) \mbox{, } \emb(-,N)
	\end{align*}
	are analytic. By applying Lemma~\ref{theorem:holim_lemma} to the homotopy pullback square from Theorem~\ref{theorem:Main_Pullback_Diagram}
	\begin{align*}
		\xymatrix{
		\emb_{\A}(-,N) \ar[r] \ar[d] & \maps(-,\A) \ar[d]  \\
		\emb(-,N) \ar[r]             & \maps(-,\gr(m,TN))
		}
	\end{align*}
	we get the analyticity of $\emb_{\A}(-,N)$.

	As $\A' \simeq \A$, when $M$ is a manifold in $\disc$, the inclusion of $\emb_{\A'}(M,N)$ into $\emb_{\A}(M,N)$ is a weak homotopy equivalence. As $\T_\infty$ is defined to be the Kan extension along the inclusion $\disc \hookrightarrow \man$, there is a natural homotopy equivalence $\T_\infty \emb_{\A'}(M,N) \xrightarrow{\simeq} \T_\infty \emb_{\A}(M,N)$.
	The second part of the theorem follows from the analyticity of $\emb_{\A}(-,N)$.
\end{proof}


\section{Lagrangian Embeddings}
\label{sec:Lagrangian_embeddings}
In this section, we use the above framework to study the Lagrangian embeddings functor.

A {\emph symplectic manifold} is a pair $(N,\omega)$ where $N$ is a smooth manifold and $\omega$ is a closed, non-degenerate, differential 2-form on $N$.
Existence of a symplectic form on $ N$ forces it to be even dimensional.
For every point $q$ in $N$, the 2-form $\omega$ restricts to a bilinear form on the tangent space $T_q N$. We say that a subspace $V$ of $T_q M$ is \emph{Lagrangian} if $\dim V = n/2$ and $\omega|_V \equiv 0$.
A submanifold $M$ of $ N$ is called \emph{Lagrangian} if  $ \omega|_{M} \equiv 0$ and $ \dim {M} = n/2$.
This is equivalent to requiring that for every point $p$ in $M$, the tangent space $T_p M$ is a Lagrangian subspace of $T_p N$.

For the rest of this section fix a symplectic manifold $(N,\omega)$ of dimension $n = 2m$.
\begin{definition}
	Let $\mathrm{Lag}$ be the subset of $\gr(m,TN)$ whose fiber over a point $q$ in $N$ is the space of Lagrangian subspaces of $T_q N$.
\end{definition}
Note that $e: M \hookrightarrow N$ is a $\mathrm{Lag}$-directed embedding if and only if $e(M)$ is a Lagrangian submanifold of $M$, so that $\embl(M,N)$ is the space of Lagrangian embeddings of $M$ into $N$.

The $h$-principle does not hold for $\mathrm{Lag}$-directed immersions and embeddings. Instead, $\mathrm{Lag}$ is homotopy equivalent to a larger space for which $h$-principles do hold.

\begin{definition}
	An \textit{almost complex structure} $J$ on $ N$ is a vector bundle isomorphism $ J :T N \rightarrow T N$ satisfying $ J^2 = -1$. This is equivalent to saying that the structure group of $TN$ can be reduced to $ \GL_{m}(\C)$.
\end{definition}

\begin{definition}
	An almost complex structure $J$ on a symplectic manifold $(N,\omega)$ is said to be \emph{compatible} with the symplectic structure if the following two conditions are satisfied.
	\begin{enumerate}
		\item $\omega(-,J-)$ defines a Riemannian metric on $N$,
		\item $\omega(J-,J-) = \omega(-,-)$.
	\end{enumerate}
\end{definition}

On every symplectic manifold $ N$, the space of all compatible almost complex structures is non-empty and contractible \cite[Chapter 13]{book:Ana_Cannas_da_Silva}.
For the rest of this section, we will fix a compatible almost complex structure $ J$ on $(N,\omega)$.

\begin{definition}
  For a point $q$ of $N$, a subspace $V$ of $T_q N$ is called \emph{totally real} if $\dim V = n/2$ and $ V + JV = T_q N$.
  A submanifold $M$ of $N$ is called \emph{totally real} if for each point $p$ in $M$, $T_p M$ is a totally real subspace of $T_p N$.
\end{definition}

Denote by $\mathrm{TR}$ the subset of $\gr(m,TN)$ consisting of totally real subspaces of $TN$.
$\embtr(M,N)$ is the space of totally real embeddings of $M$ into $N$.
Compatibility of $J$ with $\omega$ implies that there is a natural inclusion $ \mathrm{Lag} \subseteq \mathrm{TR}$ and hence $\embl(M,N) \subseteq \embtr(M,N)$.

\begin{proposition}
	\label{theorem:equiv_lag_tr}
	The inclusion $ \mathrm{Lag} \hookrightarrow \mathrm{TR}$ is a homotopy equivalence.
\end{proposition}
\begin{proof}
  For a point $q$ in $N$, we can use the almost complex structure to identify the tangent space $T_q N$ with $\C^m$.
  Under this identification, the space of all totally real subspaces equals $\GL(m,\C)/\GL(m,\R)$.

  Arnol'd \cite{Arnold_Characteristic_class_entering_in_quantization_conditions} has shown that the space of Lagrangian subspaces of $T_q N$ is homeomorphic to $U(m)/O(m)$.

  \begin{align*}
    \xymatrix{
    \GL(m,\C)/\GL(m,\R) \ar[r]
    & \mathrm{TR} \ar[d]
    \\
    & N
    }
    &&
    \xymatrix{
    U(m)/O(m) \ar[r]
    & \mathrm{Lag} \ar[d]
    \\
    & N
    }
  \end{align*}

  The result follows from the fact that the natural inclusion of $U(m)/O(m)$ into $\GL(m,\C)/\GL(m,\R)$ is a homotopy equivalence.
\end{proof}

In \cite[Section 2.4.5]{book:Gromov_Partial_Differential_Relations}, Gromov proved the following using convex integration.
\begin{theorem}[Gromov]
  \label{theorem:totally_real_embeddings}
  Let $N$ be an almost complex manifold and let $\mathrm{TR}$ be the subset of $\gr(m,TN)$ consisting of totally real subspaces of $TN$.
  The $h$-principle holds for $\mathrm{TR}$-directed embeddings.
\end{theorem}

 Theorem~\ref{theorem:totally_real_embeddings}, Proposition \ref{theorem:equiv_lag_tr}, and Theorem~\ref{theorem:Main_theorem} give us the following result.

\begin{theorem}
	\label{theorem:lagrangian_embeddings}
	Let $ n - m \ge 3$, $ n=2m$ and let $(N,\omega)$ be a symplectic manifold of dimension $n$ with a compatible almost complex structure $J$.
  Then for all manifolds $M$ in $\man$ the analytic approximation of $ \embl(M,N)$ is weakly homotopy equivalent to $ \embtr(M,N)$ via a zig-zag of maps,
	\begin{align*}
		\xymatrix{
		\embtr(M,N) \ar[r]^-{\simeq}
			& \T_\infty \embtr(M,N)
				& \T_\infty\embl(M,N). \ar_-{\simeq}[l]
		}
	\end{align*}
\end{theorem}

	As a consequence of this, we can see that in general $\embl(-,N)$ is not analytic.
	For example, there are no simply connected Lagrangian submanifolds of $(\C^n, \omega)$, where $\omega$ is the standard symplectic structure \cite{Gromov_Pseudo_holomorphic_curves_} but $S^3$ can be embedded in $\C^3$ as a totally real manifold \cite[Section 2.4.5]{book:Gromov_Partial_Differential_Relations}.
	Thus $\embl(S^3,\C^3)$ is empty but $\T_\infty \embl(S^3,\C^3)$, which is weakly homotopy equivalent to $\embtr(S^3,\C^3)$, is not.

\begin{corollary}
	$\embl(-,\C^3)$ is not analytic on the category of 3 dimensional smooth manifolds.
\end{corollary}

\begin{remark}
	Computations of $ \pi_i(\embtr(M,\C^m))$ for $i = 0,1$ for some $M$ can be found in \cite{Audin_totally_real_embeddings}, \cite{Borrelli_On_totally_real_}.
  We hope to use manifold calculus to compute higher homotopy groups.
\end{remark}

\subsection{Isotropic Embeddings}
A stronger result is true for isotropic embeddings.
An $m$-dimensional submanifold $M$ of a symplectic manifold $(N, \omega)$ is called \emph{isotropic} if  $ \omega|_{M} \equiv 0$.
Isotropic submanifolds necessarily have dimension $\le n/2$.
When the dimension $m = n/2$ the isotropic submanifolds are the Lagrangian submanifolds.
As before, let $\mathrm{Iso}$ be the subset of $\gr(m,TN)$ whose fiber over $q$ in $N$ is the space of $m$ dimensional subspaces of $T_q N $ on which $\omega$ vanishes, so that $e: M \hookrightarrow N$ is an $\mathrm{Iso}$-directed embedding if and only if $e(M)$ is an isotropic submanifold of $M$.
By the Darboux theorem~\cite{book:Ana_Cannas_da_Silva}, every point $q$ in $N$ has an open neighborhood that is symplectomorphic to an open subset of $\R^{n}$ with the standard symplectic form.
Hence, $\mathrm{Iso} \rightarrow N$ is a fiber bundle.

In~\cite[Section 12.4]{book:Eliashberg_h_principle} Eliashberg--Mishachev prove the following theorem using the method of holonomic approximation.

\begin{theorem}[Eliashberg--Mishachev]
  Let $(N, \omega)$ be a symplectic manifold of dimension $n$ and let $\mathrm{Iso}$ be the subset of $\gr(m,TN)$ consisting of isotropic subspaces of $TN$.
  If $m < n/2$, then the $h$-principle holds for $\mathrm{Iso}$-directed embeddings.
\end{theorem}

A direct application of Theorem~\ref{theorem:Main_theorem} gives us the following result.

\begin{theorem}
 Let $ n - m \ge 3$, let $ n > 2m$ be an even integer and let $(N,\omega)$ be a symplectic manifold of dimension $n$. The functor $\emb_{\mathrm{Iso}}(-,N)$ defined on the category $\man$ is analytic.
\end{theorem}


\bibliographystyle{/Users/apurv/Library/texmf/tex/latex/local/halpha}
\bibliography{/Users/apurv/Library/texmf/tex/latex/local/references.bib}

\end{document}